\documentclass{amsart}
\usepackage{amsmath, amsthm, amssymb}
\usepackage{graphics}
\usepackage{pinlabel}
\usepackage{color}
\usepackage[colorlinks,citecolor=blue,linkcolor=red]{hyperref}

\usepackage[margin=1.5in]{geometry}
\DeclareMathOperator{\Map}{Map}
\DeclareMathOperator{\PMap}{PMap}

\DeclareMathOperator{\End}{End}
\DeclareMathOperator{\Aut}{Aut}
\DeclareMathOperator{\Homeo}{Homeo}
\DeclareMathOperator{\Comm}{Comm}

\begin{document}

\newtheorem{theorem}{Theorem}[section]
\newtheorem{lemma}[theorem]{Lemma}
\newtheorem{corollary}[theorem]{Corollary}
\newtheorem{proposition}[theorem]{Proposition}
\newtheorem*{proposition*}{Proposition}
\newtheorem{question}[theorem]{Question}
\newtheorem*{answer}{Answer}
\newtheorem{problem}[theorem]{Problem}
\newtheorem*{claim}{Claim}
\newtheorem*{criterion}{Criterion}
\theoremstyle{definition}
\newtheorem{definition}[theorem]{Definition}
\newtheorem{remark}[theorem]{Remark}
\newtheorem{example}[theorem]{Example}
\newtheorem*{case}{Case}

\def\Id{\text{Id}}
\def\H{\mathbb H}
\def\Z{\mathbb Z}
\def\N{\mathbb N}
\def\R{\mathbb R}
\def\C{\mathbb C}
\def\CP{{\mathbb {CP}}}
\def\CC{\mathcal C}
\def\M{\mathcal M}
\def\F{\mathcal F}
\def\J{\mathcal J}
\def\HC{\mathcal H}
\def\P{\mathcal P}
\def\Sph{\mathbb S}
\def\P{\mathcal P}
\def\Q{\mathbb Q}
\def\L{\mathcal L}
\def\A{\mathcal A}
\def\E{\mathcal E}
\def\homeo{\textnormal{Homeo}}
\def\inte{\textnormal{int}}
\def\scl{\textnormal{scl}}
\def\link{\textnormal{link}}
\def\Out{\textnormal{Out}}
\def\TNF{\textnormal{TNF}}
\def\Cyc{\textnormal{Cyc}}
\def\supp{\textnormal{supp}}

\newcommand{\define}[1]{\emph{#1}}

\title{Isomorphisms between big mapping class groups}
\author{Juliette Bavard}
\address{Univ Rennes, CNRS, IRMAR - UMR 6625, F-35000 Rennes, France}
\email{juliette.bavard@univ-rennes1.fr}

\author{Spencer Dowdall}
\address{Department of Mathematics, Vanderbilt University}
\email{spencer.dowdall@vanderbilt.edu}

\author{Kasra Rafi}
\address{Department of Mathematics, University of Toronto}
\email{rafi@math.toronto.edu }

\begin{abstract} We show that any isomorphism between mapping class groups of 
orientable infinite-type surfaces is induced by a homeomorphism between the surfaces. 
Our argument additionally applies to automorphisms between finite-index subgroups of 
these `big' mapping class groups and shows that each finite-index subgroup has finite outer 
automorphism group. As a key ingredient, we prove that all simplicial automorphisms 
between curve complexes of infinite-type orientable surfaces are induced by 
homeomorphisms.
\end{abstract}

\maketitle

\section{Introduction}
\label{sec:intro}

All \define{surfaces} in this paper will be connected, orientable, and without boundary. A surface $S$ is said to be of \define{finite-type} if its fundamental group is finitely generated; otherwise $S$ has \define{infinite-type}.
The \define{(extended) mapping class group} of $S$ is the group $\Map(S)$ of isotopy classes of possibly orientation-reversing homeomorphisms of $S$. 
An \define{end} of $S$ is a nested choice of connected components of $S\setminus K_i$ for some compact exhaustion $K_1\subset K_2\subset\dotsb$ of $S$. More formally, the set of ends is the inverse limit $\End(S) = \varprojlim \pi_0(S\setminus K)$ over the directed (via inclusion) system of compact subsets $K$ of $S$. 
The \define{pure mapping class group} is the subgroup $\PMap(S)\le\Map(S)$ that fixes $\End(S)$ pointwise. 
We also have the index $2$ subgroups $\PMap^+\!(S)$ and $\Map^+\!(S)$ consisting of orientation-preserving elements.

In the case of finite-type surfaces, an old result of Ivanov \cite{Ivanov88-TeichModular} shows that the automorphism group of $\Map(S)$ is isomorphic to $\Map(S)$ itself; the closed case being independently obtained by McCarthy \cite{McCarthy86}. It is a related folk-theorem (implicit in \cite{Ivanov88-TeichModular} and following in most cases from \cite{BirmanLubotzkyMcCarthy} and \cite{Harer}) that, aside from low-complexity exceptions, non-homeomorphic finite-type surfaces cannot have isomorphic mapping class groups; for a full discussion and proof see \cite[Appendix~A]{Rafi-Schleimer}. Thus the group $\Map(S)$ determines the surface $S$ when $S$ has finite-type.

Here we focus on the so-called `big' mapping class groups, that is, groups $\Map(S)$ and $\PMap(S)$ for $S$ of infinite-type. Unlike mapping class groups of finite-type surfaces, these big mapping class groups have uncountably many elements and inherit a non discrete topology from the compact open topology on $\Homeo(S)$. Despite a recent growing interest in big mapping class groups (e.g., \cite{Calegari-blog,AramayonaFossasParlier,DurhamFanoniVlamis,Patel-Vlamis,Hernandez-Morales-Valdez2}), the above properties have remained open in this setting. Our main results establish them for all infinite-type surfaces. 

\begin{theorem}
\label{thm:distinguish}
Let $S_1$ and $S_2$ be infinite-type surfaces.
For $i=1,2$, let $G_i$ be a finite-index subgroup of either $\Map(S_i)$ or $\PMap(S_i)$ and let $\Phi\colon G_1\to G_2$ be any algebraic isomorphism. Then there is a homeomorphism $h\colon S_1\to S_2$ so that $\Phi(f) = h\circ f\circ h^{-1}$. In particular, $\Phi$ is automatically continuous.
\end{theorem}

Thus mapping class groups---and even their finite-index subgroups---distinguish infinite-type surfaces. 
This answers Question 1.1 and generalizes Theorem 1 in the recent paper of Patel and Vlamis \cite{Patel-Vlamis}, who treat the special case of $\PMap$ for infinite-type surfaces of finite genus at least $4$.

The \define{abstract commensurator} of $G$ is the group $\Comm(G)$ of all equivalence classes of isomorphisms $H_1\to H_2$ between finite-index subgroups of $G$, where two such isomorphisms are equivalent if they agree on a finite-index subgroup. There are natural maps $G\to \Aut(G)\to \Comm(G)$ arising from the fact that every conjugation or automorphism of $G$ is itself a commensuration. However $\Comm(G)$ is in general much larger than $\Aut(G)$; for example $\Aut(\Z) \cong \Z/2\Z$ whereas $\Comm(\Z) \cong \Q^*$ is not even finitely generated. We view $\Comm(G)$ as capturing the  `hidden' symmetries of $G$; an assertion that $\Comm(G)$ is small thus conveys a strong \emph{algebraic rigidity} that is reminiscent of superrigidity for lattices $\Gamma$ in a semisimple Lie group $G\ne\mathrm{PSL}(2,\R)$. Indeed, here work of Margulis, Mostow and Prasad (see \cite{Margulis,Zimmer}) implies that $[\Comm(\Gamma):\Gamma]<\infty$ when $\Gamma$ is nonarithmetic and that $\Comm(\Gamma)$ virtually embeds into $G$ when $\Gamma$ is arithmetic. Theorem~\ref{thm:distinguish} implies this strong algebraic rigidity for $\Map(S)$, generalizing Ivanov's result computing $\Comm(\Map(S))$ for finite-type surfaces \cite{Ivanov97-IRMN}, as well as the following consequences which, in particular, establish Conjecture 1.2 of \cite{Patel-Vlamis}.

\begin{corollary}
\label{cor:autmap}
Let $S$ be an infinite-type surface. Then
\begin{enumerate}\renewcommand{\theenumi}{\roman{enumi}}
\item\label{maincor-Aut} The natural maps $\Map(S)\to \Aut(\Map(S))\to \Comm(\Map(S))$ are isomorphisms.
\item\label{maincor-characteristic} $\PMap(S)$, $\Map^+(S)$, and $\PMap^+(S)$ are characteristic in $\Map(S)$.
\item\label{maincor-out} $\Out(G)$ is finite for every finite-index subgroup of $\Map(S)$.
\item\label{maincor-finiteindex} Finite-index subgroups of $\Map(S)$ or $\PMap(S)$ are isomorphic iff they are conjugate.
\end{enumerate}
\end{corollary}

Our proof of Theorem~\ref{thm:distinguish} follows Ivanov's approach \cite{Ivanov97-IRMN} and has two main ingredients. The first is an algebraic characterization of Dehn twists in terms of centralizers of elements (see \S\ref{sec:algchar}). This is related to the characterization of `algebraic twist subgroups' used by Ivanov \cite{Ivanov88-TeichModular} and others and further relies on a new characterization (Proposition~\ref{proposition:characterization of finitely supported elements}) of finitely-supported elements by the cardinality of their conjugacy classes.

The second ingredient comes from curve complexes. By a \define{curve} in $S$, 
we mean the equivalence class of an embedding $\Sph^1\hookrightarrow S$ of the circle that is neither nullhomotopic nor homotopic into an end of $S$, where embeddings are equivalent if they are homotopic or differ by precomposition with an orientation-reversing homeomorphism of $\Sph^1$.

A \define{multicurve} is a finite set of distinct curves that admit representative embeddings with disjoint images. 
The \define{curve complex} of $S$ is the simplicial complex $\CC(S)$ whose simplices correspond to multicurves of $S$ and face maps to inclusions of multicurves.

The curve complex of a surface was first introduced by Harvey \cite{Harvey} as a Teichm\"uller-theoretic analogue of the Tits building for symmetric spaces. 
A powerful theorem of Ivanov \cite{Ivanov97-IRMN}, Korkmaz \cite{Korkmaz}, and Luo \cite{Luo} in the finite-type setting, analogous to a fundamental theorem of Tits \cite{Tits}, states that every simplicial automorphism of $\CC(S)$ is induced by an element of $\Map(S)$.
Ivanov originally used this to give a new proof of Royden's famous theorem that $\Map(S)$ is the isometry group of the Teichm\"uller space of $S$ \cite{Royden}, and it is now known that many (indeed most) other complexes built from $S$ have automorphism group equal to $\Map(S)$ (e.g., see \cite{MR2952767} or \cite{Brendle-Margalit} and the references therein). Our final theorem extends this result to infinite-type surfaces:

\begin{theorem}
\label{thm:curvecomplex}
Let $S$ and $S'$ be surfaces and suppose $S$ has infinite-type. Then any simplicial isomorphism $\CC(S)\to \CC(S')$ is induced by a homeomorphism $S\to S'$.
\end{theorem}

Theorem~\ref{thm:curvecomplex} was independently proven in a very recent paper \cite{Hernandez-Morales-Valdez} by Hern\'andez, Morales, and Valdez. 
We give a proof 
based on finite-type exhaustions and a simple observation, already present in \cite[Lemma 1]{Ivanov97-IRMN}, that a multicurve's link in $\CC(S)$ is able to detect the components of its complement in $S$ (Lemma \ref{lem:link}).

\subsection*{Acknowledgment}

Dowdall was supported by NSF grant DMS-1711089. Rafi was supported by NSERC Discovery grant RGPIN 435885. 
The authors thank Mark Bell for helping with the proof of Theorem~\ref{thm:curvecomplex}, and Yves de Cornulier for suggesting that elements without finite support may have uncountable conjugacy classes in $\Map(S)$ (c.f. Proposition~\ref{proposition:characterization of finitely supported elements}). We would also like to thank the Chili's in Fayetteville, Arkansas, and The Punter in Cambridge, U.K., for providing the margaritas and pints that facilitated these respective conversations.

\section{Preliminaries}
\label{sec:prelims}

Let us briefly establish some terminology for dealing with an infinite-type surface~$S$. A \define{domain} $Y$ in $S$ is a connected component of $S\setminus \alpha$ for some multicurve $\alpha$; we then define $\partial Y$ to be the smallest sub-multicurve $\beta$ of $\alpha$ so that $Y$ is a component of $S\setminus \beta$. Note that domains are only defined up to isotopy and that each domain $Y$ is itself a surface.
A curve in $S$ is \define{essential} in $Y$ if its equivalence class contains an embedding that defines a curve in $Y$.
A curve and a domain are \emph{disjoint} if they have disjoint representatives; thus the curves of $\partial Y$ are disjoint from $Y$. 

\begin{definition}
\label{def:principal}
A domain $Y$ of $S$ is said to be \define{principal} if $Y$ has finite-type with $\chi(Y) \le -3$
and if every component $X$ of $S\setminus \partial Y$ with $X\ne Y$ has infinite-type.
\end{definition}

Notice that $\Map(S)$ respectively acts on the sets of curves, multicurves, and domains of $S$. We make frequent implicit use of the following result of Hern\'andez, Morales, and Valdez extending the well-known Alexander method (see \cite[\S2.3]{Primer}) to the infinite-type setting:

\begin{theorem}[Hern\'andez--Morales--Valdez {\cite{Hernandez-Morales-Valdez2}}]
\label{thm:alexander}
Let $S$ be an infinite-type surface. If $f\in \Map(S)$ fixes each curve of $S$, then $f$ is trivial in $\Map(S)$.
\end{theorem}

Accordingly, we say that $f\in \Map(S)$ has \define{finite support} if there is a finite-type domain $Y$ of $S$ such that $f$ fixes every curve disjoint from $Y$. 

\begin{lemma}
\label{lem:orientation}
If $f\in \Map(S)$ has finite support, then $f$ is orientation-preserving.
\end{lemma}
\begin{proof}
By definition, there is an infinite-type domain $Y$ such that $f(Y) = Y$ and $f$ fixes each curve in $Y$. By Theorem~\ref{thm:alexander}, $f\vert_Y$ is isotopic to the identity. Thus $f$ evidently preserves the orientation on $Y$ and, consequently, all of $S$.
\end{proof}

Following Handel and Thurston \cite[\S2]{HandelThurston}, for $X$ any surface and $f\in \Map(X)$ we write $\mathcal{O}(f)$ for the set of curves $\alpha$ of $X$ such that $\{f^k(\alpha)\mid k\in \Z\}$ is finite and write $\partial f$ for the set of curves in $\mathcal{O}(f)$ that are disjoint from all other elements of $\mathcal{O}(f)$. It is clear that $\partial f$ is a canonical set of disjoint curves in $X$ for which $f(\partial f) = \partial f$. 

\begin{definition}
\label{def:annular}
Say that $f\in \Map(S)$ is \define{multi-annular} if
\begin{itemize}
\item $f$ has finite support,
\item $f$ fixes each component of $\partial f$, and
\item $f$ fixes every curve disjoint from $\partial f$.
\end{itemize}
If $\partial f$ is a single curve, we further say that $f$ is \define{annular}.
\end{definition}

Each curve $\alpha$ of $S$ determines an associated pair $D_\alpha,D_\alpha^{-1}\in\PMap(S)$ of \define{Dehn twists} about $\alpha$ defined as follows: Cut $S$ on $\alpha$ to obtain a $2$--manifold with two boundary components, rotate one component a full revolution to the left (for $D_\alpha$) or right (for $D_\alpha^{-1}$) and re-glue; for details \cite[Chapter 3]{Primer}. The Dehn twists $D_\alpha$ and $D_\alpha^{-1}$ are distinguished from each other by the choice of an orientation on $S$; thus in writing $D_\alpha$ have implicitly specified an orientation. As the distinction is not pertinent for us, we often (e.g., in Corollary~\ref{cor:generating-twists}) consider the pair $\{D_\alpha,D_\alpha^{-1}\}$, which is well-defined irrespective of orientation. 
We call $\alpha$ a \define{pants curve} if one component of $S\setminus \alpha$ is a thrice-punctured sphere. In this case there are also \define{half-twists} $H_\alpha,H_\alpha^{-1}\in \Map(S)$ satisfying $H_\alpha^{\pm2} = D_\alpha^\pm$ and defined by fixing $\alpha$ and swapping the other two punctures in the thrice-punctured sphere component of $S\setminus \alpha$; see \cite[\S9.1.3]{Primer}. Note that $H_\alpha^{\pm}\notin \PMap(S)$. To streamline notation, for each curve $\alpha$ of $S$ we define the associated \define{twists} about $\alpha$ to be
\[T_\alpha^\pm = \begin{cases}H_\alpha^\pm,& \text{if $\alpha$ is a pants curve}\\D_\alpha^\pm,&\text{otherwise}.\end{cases}\]
For a multicurve $\beta$ with components $\beta_1,\dotsc,\beta_k$, we similarly define the associated \define{twists} $T^\pm_\beta = \prod_{i=1}^k T^\pm_{\beta_i}$ about $\beta$.
We note the following trivialities:

\begin{lemma}
\label{lem:twists}
Let $\alpha,\beta$ be multicurves on a surface $S$. Then
\begin{enumerate}
\item\label{twist-annular} $T_\alpha$ is multi-annular with $\partial(T_\alpha) = \alpha$.
\item\label{twist-commute} $T_\alpha$ and $T_\beta$ commute iff $\alpha$ and $\beta$ are disjoint.
\item\label{twist-equal} If $n,m\in \Z\setminus\{0\}$ are such that $T_\alpha^n= T_\beta^m$, then $\alpha=\beta$.
\item\label{twist-conjugate} If $\sigma(f)\in \{1,-1\}$ records whether $f\in\Map(S)$ preserves orientation, then
\[f\circ T_\alpha\circ f^{-1} = T_{f(\alpha)}^{\sigma(f)}.\]
\end{enumerate}
\end{lemma}

The following fact will play a crucial role in our proof of Lemma~\ref{lem:multi-annular}.

\begin{lemma}
\label{lem:canonical_curves}
If $f\in \Map(S)$ is nontrivial and has finite support, then $\partial f$ is a nonempty multicurve in $S$.
\end{lemma}

As $\partial f$ is clearly empty when $f$ has finite-order, it will help to first establish:

\begin{lemma}\label{lemma:infinite_order}
If $f\in \Map(S)$ is nontrivial and has finite support, then $f$ has infinite-order. Furthermore, if $Y$ is a principal domain in $S$ so that $f$ fixes every curve disjoint from $Y$ and no power of $f$ is a nontrivial product of Dehn twists about curves of $\partial Y$, then the restriction $g= f\vert_Y$ is an infinite-order element of $\Map(Y)$.
\end{lemma}

\begin{remark}
\label{rem:nice_restriction}
A domain $Y$ as in Lemma~\ref{lemma:infinite_order} may always be obtained by enlarging a finite-type domain for which $f$ fixes every curve disjoint from it. 
Further, the restriction $f\vert_Y$ is well-defined in $\Map(Y)$: Indeed, $f$ induces an automorphism of $\CC(Y)$ which, according to \cite{Luo} (and using $\chi(Y)\le -3$), is equivalent to an element of $\Map(Y)$.
\end{remark}

\begin{proof}[Proof of Lemma~\ref{lemma:infinite_order}]
Fix a particular subset $Y\subset S$ representing the domain in the statement, and let $\overline{Y}$ be its closure in $S$. We similarly let the subset $\partial Y = \overline{Y}\setminus Y$ represent the multicurve $\alpha = \partial  Y$.
Let $\Gamma = \Homeo(\overline{Y},\partial Y)$ denote the group of homeomorphisms of $\overline{Y}$ that fix $\partial Y$ pointwise, and write $\Gamma_0$ for its identity component. Also let $Y'$ be the compactification of $\overline{Y}$ obtained by `plugging' each end of $\overline{Y}$ with a point. That is, $Y'$ is a compact $2$--manifold with boundary such that $\overline{Y} = Y'\setminus P$ for some finite (and possibly empty) set $P\subset \mathrm{int}(Y')$. We then similarly have $\Gamma' = \Homeo(Y',\partial Y)$ with identity component $\Gamma'_0$. We now have (see \cite[\S2.4]{Guaschi-Pineda}) an exact sequence
\[1\longrightarrow B_{k}(Y')\longrightarrow \Gamma/\Gamma_0 \longrightarrow \Gamma'/\Gamma'_0\longrightarrow 1,\]
where $B_k(Y')$ is the braid group on $k = \left\vert P\right\vert$ strands in $Y'$. (Note that $B_k(Y')$ is trivial when $P$ is empty.) 
The group $B_k(Y')$ is torsion-free by \cite[Corollary 9]{Guaschi-Pineda} (see also \cite[Theorem 8]{Fadell-Neuwirth}), and the quotient $\Gamma'/\Gamma'_0$ is torsion-free by \cite[Corollary~7.3]{Primer}. Therefore the middle group $\Gamma/\Gamma_0$ is torsion-free as well.

Since $f$ fixes every curve disjoint from $Y$, we may use Theorem~\ref{thm:alexander} to choose a representative $\varphi\in \Homeo(S)$ that restricts to the identity on $S\setminus Y$. In particular, $\varphi$ fixes $\partial Y$ pointwise. Restricting to $\overline{Y}$ now yields an element $\mu = \varphi\vert_{\overline{Y}}\in \Gamma$ such that the further restriction of $\mu$ to $Y = \mathrm{int}(\overline{Y})$ represents $g = f\vert_Y\in \Map(Y)$.

We caution that the coset of $\mu$ in $\Gamma/\Gamma_0$ is not canonically defined, as it depends on the chosen representative $\varphi$. Nevertheless, $\mu$ is nontrivial in $\Gamma/\Gamma_0$, as otherwise a path from $\mu$ to $\mathrm{Id}_{\overline{Y}}$ in $\Gamma$ would extend to an isotopy between $\varphi$ and $\mathrm{Id}_{S}$, contradicting the nontriviality of $f$.
Thus $\mu\Gamma_0\in \Gamma/\Gamma_0$ has infinite-order.

We now prove that $g =f\vert_{Y}$ has infinite-order in $\Map(Y)$; as $f^n\vert_Y = g^n$, this will imply that $f$ has infinite-order as well. If instead $g^k \simeq \mu^k\vert_Y$ is trivial for $k\ge 1$, then we may adjust $\mu^k$ by an isotopy in $Y = \mathrm{int}(\overline{Y})$ to obtain some $\psi\in \Gamma$ that is supported in a neighborhood of $\partial Y$ and is in fact a nontrivial (since $\mu^k\Gamma_0\ne\Gamma_0)$ product of Dehn twists about the curves of $\partial Y$; see \cite[Proposition 3.19]{Primer}. Extending this isotopy $\mu^k\simeq \psi$ via the identity gives an isotopy from $f^k\simeq \varphi^k$ to a nontrivial element of the form 
\[
D_{\gamma_1}^{k_1}\dotsc D_{\gamma_n}^{k_n}\in \Map(S),
\] 
where $\gamma_1,\dotsc, \gamma_n$ are the component curves of $\partial Y$. This contradicts our assumption on $f$.
\end{proof}

\begin{proof}[Proof of Lemma~\ref{lem:canonical_curves}]
Fix an exhaustion $Y_1\subset Y_2\subset \dotsc$ of $S$ by domains $Y_i$ satisfying the hypothesis of Lemma~\ref{lemma:infinite_order} and such that $\partial Y_i$ is essential in $Y_{i+1}$ for each $i$. Let $g_i \in \Map(Y_i)$ be the restriction $f\vert_{Y_i}$ to $Y_i$ (see Remark~\ref{rem:nice_restriction}) and note that $g_i$ has infinite-order by Lemma~\ref{lemma:infinite_order}.

Consider the sets $\mathcal{O}(g_i)$ and $\partial g_i$. Since $g_i$ has infinite-order and $\mathcal{O}(g_{i+1})$ is nonempty (as it contains $\partial Y_i$), we may apply \cite[Lemma 2.2]{HandelThurston}
to conclude that $\partial g_i$ is nonempty for each $i > 1$. Note also that $\partial g_i$ is finite, as $Y_i$ has finite-type.
It is clear from the definitions that $\mathcal{O}(g_i) \subset\mathcal{O}(g_{i+1})$ for each $i$ and that 
\[
\mathcal{O}(f) = \bigcup_i \mathcal{O}(g_i)
\qquad\text{and}\qquad
\partial f = \bigcup_i \bigcap_{j\ge i} \partial g_j.
\] 
Since $\mathcal{O}(g_{i+1})$ contains all curves of $Y_{i+1}$ that are disjoint from $Y_i$ by construction, we see that each element of $\partial g_{i+1}$ must in fact be an essential curve of $Y_i$. Therefore we have $\partial g_{i+1} \subset \partial g_{i}$ and may consequently conclude that $\partial f = \cap_i \partial g_i$ is a nonempty finite set of disjoint curves of $S$.
\end{proof}

\section{Automorphisms of curve complexes}
\label{sec:curvecomplex}

In this section we prove Theorem~\ref{thm:curvecomplex}. If $\alpha$ is a multicurve in a surface $S$, the \define{link of $\alpha$} is the full subcomplex $\link(\alpha)\subset \CC(S)$ spanned by the set of vertices of $\CC(S)\setminus\alpha$ that are adjacent to $\alpha$ (that is, the curves $\beta$ that are distinct and disjoint from each curve of $\alpha$). Define a relation $\sim$ on the vertices of $\link(\alpha)$ by declaring $\beta\sim\delta$ if there exists a vertex in $\link(\alpha)$ that is nonadjacent to both $\beta$ and $\delta$. For $\beta$ a vertex of $\link(\alpha)$, we denote by $[\beta]$ the set of curves related to $\beta$, and write $\link(\alpha)\vert_{[\beta]}$ for the full subcomplex of $\link(\alpha)$ spanned by $[\beta]$.
The following shows that $\sim$ is an equivalence relation and gives a bijection between the equivalence classes of $\link(\alpha)$ and the components of $S\setminus\alpha$ that are not thrice-punctured spheres (as such components have no essential curves).
\begin{lemma}
\label{lem:link}
Let $\alpha$ be a multicurve of an infinite-type surface $S$. Let $\beta$ be a vertex of $\link(\alpha)$, and let $Y$ be the component of $S\setminus\alpha$ containing $\beta$. Then $[\beta]$ is equal to the set of curves that are essential in $Y$ and $\link(\alpha)\vert_{[\beta]} = \CC(Y)$.
\end{lemma}
\begin{proof}
Each vertex of $\link(\alpha)$ corresponds to a curve disjoint from $\alpha$ and so lies in some connected component of $S\setminus\alpha$. If $\delta$ and $\gamma$ are nonadjacent vertices in $\link(\alpha)$, then their corresponding curves intersect and so necessarily lie in the same component. In particular, if $\gamma$ is nonadjacent to both $\delta$ and $\beta$, then $\delta$ (and $\gamma$) and $\beta$ lie in the same component of $S\setminus\alpha$. This proves that the curves of $[\beta]$ lie in $Y$. Conversely, for any curve $\delta$ contained in $Y$ we may choose a third curve $\gamma$ in $Y$ that intersects both $\delta$ and $\beta$. Thus $\delta\sim \beta$ and we have proven that $[\beta]$ is the set of curves in $Y$. The fact that $\link(\alpha)\vert_{[\beta]} = \CC(Y)$ is now immediate from the definitions.
\end{proof}

We now prove that isomorphisms of curve complexes are geometric.
\begin{proof}[Proof of Theorem~\ref{thm:curvecomplex}]
Let $\Psi\colon \CC(S)\to \CC(S')$ be an isomorphism.
Fix an exhaustion $Y_1\subset Y_2\subset\dotsc$ of $S$ by principal domains $Y_i$ (Definition~\ref{def:principal}) and set $\alpha_i = \partial Y_i$. By enlarging the domains if necessary, we assume the curves of $\alpha_i$ are essential in $Y_{i+1}$. Since $Y_i$ is principal, Lemma~\ref{lem:link} implies that the equivalence class $E_i$ corresponding to $Y_i$ is the unique equivalence class of $\link(\alpha_i)$ with finite clique number.

For each $i$, we set $\alpha'_i = \Psi(\alpha_i)$ and observe that $\Psi$ restricts to an isomorphism $\link(\alpha_i)\to \link(\alpha'_i)$ that maps equivalence classes to equivalence classes. The image $E'_i$ of $E_i$ is therefore the unique equivalence class of $\link(\alpha'_i)$ with finite clique number. Writing $Y'_i$ for the component of $Y'_i$ of $S'\setminus\alpha'_i$ corresponding to $E'_i$ (Lemma~\ref{lem:link}), it follows that $Y'_i$ has finite-type and that $\Psi$  restricts to an isomorphism
\[\CC(Y_i) \cong \link(\alpha_i)\vert_{E_i}\stackrel{\Psi}{\longrightarrow}\link(\alpha'_i)\vert_{E'_i}\cong \CC(Y'_i).\]
By the original result for finite-type curve complexes (e.g., \cite{Luo}), each of these isomorphisms is induced by a homeomorphism 
\[
\phi_i\colon Y_i\to Y'_i\subset S'.
\]
We note that $\phi_{i+1}$ is compatible with $\phi_i$ by construction. That is, $\phi_{i+1}(Y_i)=Y'_i$ with the restriction of $\phi_{i+1}$ to $Y_i$ agreeing with $\phi_i$. Since $S$ is the union of the $Y_i$, the direct limit of $(\phi_i)$ now gives a homeomorphism $\phi\colon S \rightarrow S'$ inducing $\Psi$.
\end{proof}

\section{Algebraic characterization of twists}
\label{sec:algchar}
For the entirety of this section, fix an infinite-type surface $S$ and let $\Gamma$ denote either $\Map(S)$ or $\PMap(S)$. Fix also a finite-index subgroup $G$ of $\Gamma$. Our goal in this section is to give an algebraic characterization of certain `generating twists' of~$G$ (Definition~\ref{def:gen-twist}). The first step is to characterize finitely-supported elements:

\begin{definition}
\label{def:countable_conj}
Set $\F_G = \{g\in G \mid \text{the conjugacy class of $g$ in $G$ is countable}\}\le G$.
\end{definition}

\begin{proposition}\label{proposition:characterization of finitely supported elements}
An element $f\in G$ has finite support if and only if $f\in \F_G$.
\end{proposition}
\begin{proof}
Assume f does not have finite support. Then there exists a curve $a_1$ such that $f(a_1) \ne a_1$. Now suppose we have chosen distinct disjoint curves $a_1, \ldots a_n$ such that, for every $1 \leq i \leq n$, $b_i = f(a_i)$ is distinct from all $a_j$ and so that the curves $a_i$ and $b_j$ are disjoint except possibly when $i=j$. Then take a finite-type domain that contains the curves $a_i$, $b_i=f (a_i)$, and $f^{-1}(a_i)$ for $i =1,\ldots,n$. Since $f$ has infinite support, we can find new curve $a_{n+1}$ outside of $Y$ that is not fixed. By induction, we thus get an infinite list of curves $a_i$ not fixed by $f$, with the property that all $a_i$ and $b_j=f(a_j)$ are distinct and disjoint except maybe when $i=j$. 

Since $G$ has finite-index in $\Gamma$, for each $i$ we may choose $k_i\ge 1$ so that $T_{a_i}^{k_i}\in G$. For each sequence $\epsilon = (\epsilon_1,\epsilon_2\dotsc)$ with $\epsilon_i\in \{1,-1\}$, we consider the infinite product
\[\phi_\epsilon = \prod_i T_{a_i}^{\epsilon_i k_i}\in G.\]
The associated conjugates $f_\epsilon = \phi_\epsilon^{-1} f \phi_\epsilon$ are then all distinct. Indeed, if $\epsilon' = (\epsilon'_1,\epsilon'_2,\dotsc)$, then our choice of $a_i$ and $b_i = f(a_i)$ allows us to easily observe that
\[
\phi_{\epsilon'}(f_\epsilon f_{\epsilon'}^{-1})\phi_{\epsilon'}^{-1} 
= \phi_{\epsilon'} \phi_{\epsilon}^{-1} (f \phi_{\epsilon} f^{-1}) (f\phi_{\epsilon'}^{-1} f^{-1})
= \prod_{\{i \mid \epsilon_i \ne \epsilon'_i\}} T_{a_i}^{(\epsilon'_i - \epsilon_i)k_i} T_{b_i}^{\sigma(f)(\epsilon_i - \epsilon'_i)k_i}
\]
is nontrivial when $\epsilon\ne\epsilon'$. Therefore the conjugacy class of $f$ in $G$ is uncountable.

Conversely, every finitely supported mapping class may be written as a finite product of Dehn twists and half-twists (see, e.g., \cite[Corollary 4.15]{Primer}). As there are only countably many curves, it follows that $\Map(S)$ has only countably many finitely supported elements. Therefore, when $f$ has finite support, its conjugacy class in $G$ is countable.
\end{proof}

Given an element $f\in G$, we write 
\[
C_G(f) = \big\{g\in G \mid gf = fg \big\}\le G
\]
for the centralizer of $f$ and write $Z\big(\F_G\cap C_G(f)\big)$ for the center of the subgroup $\F_G\cap C_G(f)$. The following notation will help us algebraically identify twists:

\begin{definition}
\label{def:set-M_G}
Write $\mathcal{M}_G\subset G$ for the set of of elements $f\in G$ satisfying
\begin{enumerate}
\item\label{MG-finite_support} $f\in \F_G$,
\item\label{MG-cyclic_center} $Z\big(\F_G\cap C_G(f)\big)$ is infinite cyclic, and
\item\label{MG-stable_centralizer} $C_G(f) = C_G(f^k)$ for all $k\ge 1$.
\end{enumerate}
For each $f\in \mathcal{M}_G$, set $\left(\mathcal{M}_G\right)_f = \{h\in \mathcal{M}_G \mid fh = hf\}$.
\end{definition}

\begin{lemma}
\label{lem:annular_structure}
Let $f\in G$ be annular and consider the twist $T_\alpha$ (i.e., Dehn twist or half-twist) about the curve $\alpha= \partial f$.
Then 
\[
Z\big(\F_G\cap C_G(f)\big) = \langle T_\alpha\rangle \cap G\cong \Z
\qquad\text{and}\qquad
C_G(f) = C_G(f^k)
\]
for each $k\ge 1$. In particular, $f\in \langle T_\alpha\rangle$ and furthermore $f\in \M_G$.
\end{lemma}
\begin{proof}
Choose $j\ge 1$ so that $T_\alpha^j$ generates $\langle T_\alpha\rangle \cap G$. First observe that because $f$ is annular, it is a power of $T_\alpha$. Indeed, if $\alpha$ is not a pants curve, then according to Alexander's method in its finite and infinite versions (see Theorem~\ref{thm:alexander}), $f$ is homotopic to the identity on each component of $S\setminus\alpha$, thus $f$ is a non-zero power of $D_\alpha$; if $\alpha$ is a pants curve, then $f$ is homotopic to the identity on one component of $S\setminus\alpha$, and the other component is a three punctured sphere, on which $f$ is either homotopic to the identity or $f$ is a non-zero power of a half-twist. In both cases, $f = T_\alpha^{jm}$ for some $m\in \Z\setminus\{0\}$.

By Lemma~\ref{lem:twists}(\ref{twist-conjugate})
 we have for each $k\in \Z\setminus\{0\}$ that
\[
C_G(T_\alpha^j) 
= \big\{g\in G \mid \text{$g(\alpha) = \alpha$ and $g$ preserves orientation}\big\} 
= C_G(T_\alpha^{jk}).\]
Therefore 
\[
C_G(f) = C_G(T_\alpha^{jm}) = C_G(f^k)
\] 
for each $k\ge 1$. 
Since $T_\alpha^j\in \F_G\cap Z(C_G(f))$, we clearly have 
\[
T_\alpha^j\in Z\big(\F_G\cap C_G(f)\big).
\]
Conversely, let $g\in Z\big(\F_G\cap C_G(f)\big)$ be nontrivial. 
Then $g$ has finite support by Proposition~\ref{proposition:characterization of finitely supported elements}. 
If $g$ is not annular with $\partial g = \alpha$, then by definition there is a curve $\beta$ in $S\setminus\alpha$ with $g(\beta)\ne \beta$. But then $D_\beta^i\in \F_G\cap C_G(f)$ for some $i$ by the above and $g D_\beta^i g^{-1}\ne D_\beta^i$ by Lemma~\ref{lem:twists}; contradicting our choice of $g$. Therefore $g$ must be annular with $\partial g = \alpha$; by the above this implies $g\in \langle T_\alpha^j\rangle$ and so proves $Z(\F_G\cap C_G(f)) = \langle T_\alpha^j\rangle\cong \Z$.
\end{proof}

\begin{lemma}\label{lem:multi-annular}
Let $f$ be an element of $G$. If $f\in \M_G$, then $f$ is multi-annular.
\end{lemma}
\begin{proof}
Since $f\in \F_G$, we know that $f$ has finite support and, by Lemma~\ref{lem:canonical_curves}, that $\alpha = \partial f$ is a nonempty multicurve. Consider the twist $T_\alpha$ about $\alpha$. Let $g\in \F_G\cap C_G(f)$ be arbitrary. Then $g$ preserves orientation (Proposition~\ref{proposition:characterization of finitely supported elements} and Lemma~\ref{lem:orientation}) and we have:
\[g(\partial f)=\partial(gfg^{-1})=\partial f.\]
Thus $g$ commutes with $T_\alpha$ by Lemma~\ref{lem:twists}(\ref{twist-conjugate}), showing that $T_\alpha$ is in $Z\big(\F_G\cap C_G(f)\big)$. 
Since $f\in Z\big(\F_G\cap C_G(f)\big)$ as well and this group is infinite cyclic by assumption, there necessarily exist $m,n \geq 1$ so that $f^m=T_\alpha^n$.

We claim that $f$ is multi-annular. First, to see that $f$ fixes each curve comprising $\alpha$, let $\gamma$ be one such curve and choose $k\ge 1$ so that $f^k(\gamma) = \gamma$; this is possible since $f$ permutes the finitely many curves of $\alpha$. Then $f^k$ commutes with $T_{\gamma}$ by Lemma~\ref{lem:twists}. Choosing $j\ge 1$ so that $T_\gamma^j\in G$, it follows that 
\[
T_\gamma^j\in C_G(f^k) = C_G(f).
\] 
But this is only possible if $f(\gamma) = \gamma$, as required.

It remains to show that $f$ fixes each curve disjoint from $\alpha$. Let $\beta$ be one such curve and choose $i\ge 1$ so that $T_\beta^i\in G$. Since $\beta$ and $\alpha$ are disjoint, we then have 
\[
T_\beta^i\in C_G(T_\alpha^n) = C_G(f^m) = C_G(f).
\] 
Hence, again by Lemma~\ref{lem:twists}, we have $f(\beta) = \beta$.
\end{proof}

\begin{proposition}
\label{propr:char_annular}
An element $f\in G$ is annular if and only if $f\in \M_G$ and $\left(\M_G\right)_f$ is a maximal (w.r.t. inclusion) member of the collection $\{\left(\M_G\right)_h\}_{h\in \M_G}$.
\end{proposition}
\begin{proof}
First suppose $f$ is annular and let $\alpha = \partial f$.  We have seen (Lemma~\ref{lem:annular_structure}) that $f\in \M_G$. Let $h\in \M_G$ be such that $\left(\M_G\right)_f \subset (\M_G)_h$. Let $\beta$ be any curve disjoint from $\alpha$ and choose $k\ge 1$ so that $T_\beta^k\in G$. Then $T_\beta^k\in \M_G$ and evidently $T_\beta^k\in (\M_G)_f$. By assumption, this gives $hT_\beta^k = T_\beta^kh$, thus $h(\beta) = \beta$ by Lemma~\ref{lem:twists}. Therefore $h$ fixes every curve disjoint from $\alpha$, proving that $h$ is annular with $\partial h = \alpha$. It now follows from Lemma~\ref{lem:annular_structure} that $f^m = h^n$ for some $m,n\in \Z$. Thus we may conclude the desired maximality of $(\M_G)_f$ by noting
\[(\M_G)_f = C_G(f)\cap \M_G = C_G(f^m=h^n)\cap \M_G = C_G(h)\cap \M_G = (\M_G)_h.\]

Next suppose $f\in \M_G$ and that $f$ is not annular. Then $\partial f$ contains two distinct curves $\delta$ and $\gamma$. Pick a curve $\beta$ that intersects $\delta$ but is disjoint from $\gamma$. Choose $k\ge 1$ so that $T_\gamma^k,T_\beta^k\in G$ and consequently $T_\gamma^k,T_\beta^k\in \M_G$. Let $h\in (\M_G)_f$ be arbitrary. Then $h(\partial f) = \partial f$ so we may choose a power $h^i$ that fixes each component of $\partial f$. In particular, we have $h^i(\gamma) = \gamma$ so that $hT_\gamma = T_\gamma h$. Thus $h\in (\M_G)_{T_\gamma^k}$ and we have proven 
\[
(\M_G)_f\subset (\M_G)_{T_\gamma^k}.
\] 
However, $T_\beta^k$ lies in $(\M_G)_{T_{\gamma}^k}$ (since $\gamma$ and $\beta$ are disjoint) but not in $(\M_G)_{f}$ (since, e.g., the orbit of $\delta\subset \partial f$ under $T_\beta^k$ is infinite). Thus $(\M_G)_f$ is not maximal.
\end{proof}

\begin{definition}[Generating twist]
\label{def:gen-twist}
Say that $f\in G$ is a \define{generating twist} of $G$ if 
\begin{enumerate}
\item  $f\in \F_G$, 
\item $Z(\F_G\cap C_G(f))$ is infinite cyclic and generated by $f$, 
\item $C_G(f) = C_G(f^k)$ for all $k\ge 1$, and 
\item $(\M_G)_f$ is maximal in the collection $\{(\M_G)_h\}_{h\in \M_G}$. 
\end{enumerate} 
Note that these are algebraic conditions in terms of the group structure of $G$.
\end{definition}

The following is a now consequence of Lemmas~\ref{lem:twists} and \ref{lem:annular_structure} and Proposition \ref{propr:char_annular}.

\begin{corollary}
\label{cor:generating-twists}
For each curve $\alpha$ of $S$ there is a unique $j_\alpha\ge 1$ so that $T_\alpha^{j_\alpha}$ and $T_\alpha^{-j_\alpha}$ are  generating twists of $G$. This assignment $\alpha \mapsto \{T_\alpha^{\pm j_\alpha}\}$ gives a bijection between curves and inverse pairs of generating twists under which two curves are disjoint if and only if their associated generating twists commute.
\end{corollary}

\section{Isomorphisms between big mapping class groups}
We may now easily prove our main results:

\begin{proof}[Proof of Theorem~\ref{thm:distinguish}]
For $i=1,2$ let $S_i$ be an infinite type surface and $G_i$ a finite-index subgroup of $\PMap(S_i)$ or $\Map(S_i)$. For each curve $\alpha$ of $S_1$, let $T_\alpha^{j_\alpha}$ be the associated generating twist from Corollary~\ref{cor:generating-twists}. Since generating twists are defined algebraically, they are preserved by the given isomorphism $\Phi\colon G_1\to G_2$. Therefore, for each curve $\alpha$ of $S$ we have
\begin{equation}
\label{eqn:twist-image}
\Phi(T_\alpha^{j_\alpha}) = T_{h(\alpha)}^{i_\alpha}\tag{$\ddagger$}
\end{equation}
 for some unique curve $h(\alpha)$ of $S_2$ and power $i_\alpha\in \Z\setminus\{0\}$. Since the isomorphism $\Phi$ preserves commutativity, Corollary~\ref{cor:generating-twists} ensures that $\alpha$ and $\beta$ are disjoint if and only if $h(\alpha)$ and $h(\beta)$ are disjoint. The assignment $\alpha\mapsto h(\alpha)$ thus extends to a simplicial automorphism $\CC(S_1)\to \CC(S_2)$ and is consequently, by Theorem~\ref{thm:curvecomplex}, induced by some homeomorphism $h\colon S_1\to S_2$. 

We show, for each $f\in G_1$, that
\[
\Phi(f) = h\circ f\circ h^{-1}\colon S_2\to S_2
\]
Following \cite[Section 3]{Ivanov97-IRMN}, for each $f\in G_1$ and curve $\alpha$ of $S_1$, \eqref{eqn:twist-image} and Lemma~\ref{lem:twists}(\ref{twist-conjugate}) give
\begin{align*}
\Phi(f T_\alpha^{j_\alpha} f^{-1}) 
= \Phi(f) \Phi(T_\alpha^{j_\alpha}) \Phi(f^{-1})
= \Phi(f) T_{h(\alpha)}^{i_\alpha} \Phi(f)^{-1} 
= T_{\Phi(f)(h(\alpha))}^{\sigma(\Phi(f))i_\alpha}
\end{align*}
and similarly
\begin{align*}
\Phi(f T_\alpha^{j_\alpha} f^{-1}) 
= \Phi(T_{f(\alpha)}^{j_\alpha})
= T_{h(f(\alpha))}^{i_{f(\alpha)}}.
\end{align*}
Since twists have a common power if and only if their supporting curves agree (Lemma~\ref{lem:twists}(\ref{twist-equal})), this proves $\Phi(f)(h(\alpha)) = h(f(\alpha))$ for all curves $\alpha$ and all $f\in G_1$. Applying this with $\alpha = h^{-1}(\beta)$, we conclude that 
\[
\Phi(f)(\beta) = h\circ f \circ h^{-1}(\beta)
\] 
for every curve $\beta$ of $S_2$. Therefore $\Phi(f) = h \circ f \circ h^{-1}$ by Theorem~\ref{thm:alexander}, as claimed.
\end{proof}

\begin{proof}[Proof of Corollary~\ref{cor:autmap}]
For (\ref{maincor-Aut}), let $\hat{\iota}\colon \Aut(\Map(S))\to\Comm(\Map(S))$ be the natural map sending an automorphism to its equivalence class of commensurations, and let 
\[
\iota\colon \Map(S)\to \Aut(\Map(S))
\]
be the homomorphism sending $f$ to \mbox{$g\mapsto f g f^{-1}$}. If $f\in \ker(\hat{\iota}\circ\iota)$, then there is a finite index subgroup $G\le\Map(S)$ such that $\iota(f)\vert_G$ is the identity. Then for every curve $\alpha$ we may choose $n\ge 1$ so that $T^n_\alpha\in G$ and consequently
\[
T^n_\alpha = \iota(f)(T^n_\alpha)= f T^n_\alpha f^{-1} = T^{n\sigma(f)}_{f(\alpha)}.
\]
Thus $f$ is trivial by Lemma~\ref{lem:twists}(\ref{twist-equal}) and Theorem~\ref{thm:alexander}, showing that $\hat{\iota}\circ\iota$ is injective. On the other hand, for each isomorphism $\Phi\colon G\to G'$ of finite-index subgroups, Theorem~\ref{thm:distinguish} provides $h\in \Map(S)$ so that $\Phi = \iota(h)\vert_{G}$, showing that $\iota$ and $\hat{\iota}\circ\iota$  are surjective as well. For (\ref{maincor-characteristic}), since every automorphism of $\Map(S)$ is inner, the normality of these subgroups implies they are characteristic. For (\ref{maincor-out}), Theorem~\ref{thm:distinguish} gives a surjection
\[N(G)/G\to \Aut(G)/\mathrm{Inn}(G) = \Out(G),\]
where $N(G)$ is the normalizer of $G$ in $\Map(S)$. Thus when $G$ has finite-index, $[N(G):G]$ and $\Out(G)$ are finite. Finally, (\ref{maincor-finiteindex}) is a special case of Theorem~\ref{thm:distinguish}.
\end{proof}

\bibliographystyle{alphanum}
\bibliography{bigMap}

\end{document}